\documentclass[12pt]{article}
\usepackage{dsfont}
\usepackage{IEEEtrantools}
\usepackage{mathtools, nccmath}
\usepackage{latexsym}
\usepackage{amsfonts}
\usepackage{amssymb}
\usepackage{graphicx}
\usepackage{amsmath,amsfonts,amsthm}
\usepackage{amsmath}
\usepackage{mathrsfs}
\usepackage{enumerate}
\usepackage{dsfont}
\usepackage{hyperref}
\usepackage{pst-node}
\usepackage{cleveref}
\usepackage{tikz-cd}
\usepackage{nicematrix}

\usepackage{multirow}
\usepackage{tikz}
\usetikzlibrary{arrows,decorations.markings}
\usepackage{xcolor}
\usepackage{bookmark}
\usepackage{hyperref}
\hypersetup{
     colorlinks   = true,
     citecolor    = blue
}

\newcommand\restr[2]{{
  \left.\kern-\nulldelimiterspace 
  #1 
  \littletaller 
  \right|_{#2} 
  }}

\newcommand{\littletaller}{\mathchoice{\vphantom{\big|}}{}{}{}}

\newcommand{\be}{\begin{equation}}
\newcommand{\ee}{\end{equation}}
\newcommand{\bea}{\begin{eqnarray}}
\newcommand{\eea}{\end{eqnarray}}
\newcommand{\bean}{\begin{eqnarray*}}
\newcommand{\eean}{\end{eqnarray*}}
\newcommand{\brray}{\begin{array}}
\newcommand{\erray}{\end{array}}
\newcommand{\biearray}{\begin{IEEEarray}{rCl}}
\newcommand{\eiearray}{\end{IEEEarray}}

\newtheorem{dfn}{Definition}[section]
\newtheorem{thm}[dfn]{Theorem}
\newtheorem{lmma}[dfn]{Lemma}
\newtheorem{ppsn}[dfn]{Proposition}
\newtheorem{crlre}[dfn]{Corollary}
\newtheorem{xmpl}[dfn]{Example}
\newtheorem{rmrk}[dfn]{Remark}

\newcommand{\bdfn}{\begin{dfn}\rm}
\newcommand{\bthm}{\begin{thm}}
\newcommand{\blmma}{\begin{lmma}}
\newcommand{\bppsn}{\begin{ppsn}}
\newcommand{\bcrlre}{\begin{crlre}}
\newcommand{\bxmpl}{\begin{xmpl}}
\newcommand{\brmrk}{\begin{rmrk}\rm}

\newcommand{\edfn}{\end{dfn}}
\newcommand{\ethm}{\end{thm}}
\newcommand{\elmma}{\end{lmma}}
\newcommand{\eppsn}{\end{ppsn}}
\newcommand{\ecrlre}{\end{crlre}}
\newcommand{\exmpl}{\end{xmpl}}
\newcommand{\ermrk}{\end{rmrk}}

\def \qed { \mbox{}\hfill
$\Box$\vspace{1ex}}

\title{Sections and Chapters}

\begin{document}
\author{Keshab Chandra Bakshi\;and\;Silambarasan C}

\title{Unitary normalizers in finite-dimensional inclusions}

\maketitle
\begin{abstract}
    We study regular inclusions of finite-dimensional von Neumann algebras from a matrix-theoretic perspective. To this end, we introduce a new combinatorial invariant of an inclusion, called the \textit{normalizer matrix}, which encodes the structure of the normalizer purely at the level of the inclusion matrix. Using this invariant, we obtain a complete characterization of regular inclusions of finite-dimensional von Neumann algebras. As consequences, we show that every regular inclusion decomposes into finite direct sums and tensor products of basic building blocks, and that regular inclusions are necessarily of depth two. We further investigate the existence of unitary orthonormal bases in the sense of Pimsner–Popa and prove that, under a natural spectral condition, regularity is equivalent to the existence of a unitary orthonormal basis contained in the normalizer. These results provide a unified description of regularity, unitary bases, and depth through the normalizer matrix formalism.
\end{abstract}

\section{Introduction}

Normalizers occupy a central position in the structure theory of
operator algebras, reflecting internal symmetries of inclusions and
governing their associated crossed-product and groupoid constructions.
Their importance is already evident in Dixmier’s seminal classification
of maximal abelian subalgebras (MASAs) into singular, semi-regular, and
regular (Cartan) types, where regularity is defined by generation
of the ambient algebra by the normalizer. This perspective was further
developed by Feldman and Moore \cite{FM}, who showed that Cartan subalgebras give
rise to measured equivalence relations, and later, Renault (\cite{Ren}) discovered a $C^*$-algebraic analogue of this result in the language of groupoids. 

\medskip

In subfactor theory, normalizers play an important role (see, for instance, \cite{JP,C, BG2, BG, PopaVaes, CH}). Regular subfactors form a distinguished class encompassing crossed-product inclusions by finite groups and are closely related to depth-two phenomena and quantum symmetry. 
In particular, finite-index regular subfactors are known to admit unitary Pimsner--Popa bases and to arise from actions of weak Hopf algebras or quantum groupoids (see \cite{BG,CKP}).
Related classification results have also been obtained in \cite{BG5} for finite-index regular inclusions of simple unital $C^*$-algebras, where regularity forces depth two and a crossed-product-type structure.

Motivated by these developments, the present article undertakes a systematic study of {regular inclusions of finite-dimensional von Neumann algebras}.
Despite the simplicity of the finite-dimensional setting, a complete and purely structural characterization of regularity has so far been missing.
Our goal is to provide such a classification and to relate regularity to combinatorial invariants of the inclusion and to the existence of unitary orthonormal bases.

Our approach is matrix-theoretic in nature.
Recall that any inclusion $\mathcal{B} \subset \mathcal{A}$ of finite-dimensional von Neumann algebras is completely determined, up to isomorphism, by its inclusion matrix together with the dimension vectors of the simple summands.
We introduce a new combinatorial invariant of an inclusion, called the normalizer matrix (see \cref{normalizer matrix}), which captures the structure of the normalizer purely at the level of the inclusion matrix.
This invariant isolates the precise combinatorial obstructions to regularity and allows us to characterize regular inclusions in terms of elementary matrix data.

Our first main result provides a complete characterization of regularity in finite dimension.

\medskip
\noindent\textbf{Theorem A:} (See Theorem \ref{mainthm})
\emph{The inclusion of finite dimensional von Neumann algebras, $\mathcal{B} \subset \mathcal{A}$ is regular if and only if}
\begin{enumerate}\itemsep0em
    \item \emph{The inclusion matrix is an \emph{normalizer matrix}, and}
    \item \emph{For each $0 \le i \le s-1$, if $j$ is in the $i$-th row support of the inclusion matrix, then the corresponding summands of $\mathcal{B}$ have equal dimension.}
\end{enumerate}
\emph{As a consequence, any regular inclusion is isomorphic to tensor product and direct sum of inclusions $M_k(\mathbb{C}) \subset M_k(\mathbb{C}),$  $\mathbb{C} \subset \oplus_{i=0}^{v-1} M_{l_i}(\mathbb{C})$ and $\mathbb{C}^t \subset M_t(\mathbb{C})$.}

\medskip

A second motivation for this work comes from the theory of unitary orthonormal (Pimsner--Popa) bases. Such bases play a central role in subfactor theory \cite{Popa} and quantum information theory, where they are closely related to unitary error bases and to protocols such as teleportation and dense coding.
While the existence of unitary orthonormal bases is automatic for several basic inclusions, a complete characterization for general finite-dimensional inclusions has remained open. In a seminal work, Werner
\cite{Wer} established a one-to-one correspondence between tight
teleportation schemes for tripartite systems
$M_n(\mathbb C)\otimes M_n(\mathbb C)\otimes M_n(\mathbb C)$ and
orthonormal bases of unitaries in $M_n(\mathbb C)$. This correspondence
was subsequently extended by Conlon et al.~\cite{CCKL} to inclusions of
the form $N\subset M_n(\mathbb C)$, demonstrating how unitary
orthonormal bases on general multi-matrix algebras can be used to study
quantum teleportation schemes and quantum graphs. In these settings, tight
teleportation schemes are shown to arise from orthonormal unitary
Pimsner–Popa bases, generalizing Werner’s original framework. In all
such applications, it is essential that the unitary orthonormal basis
be contained in the normalizer.

In this direction, we precisely characterize when a regular inclusion admits a unitary orthonormal basis contained in the normalizer. Unlike in the subfactor setting, regularity by itself does not ensure the existence of a unitary orthonormal basis in the normalizer; an additional spectral condition on the inclusion matrix, introduced in \cite{BB}, is required.
Our second main result establishes a sharp equivalence.

\medskip
\noindent\textbf{Theorem B:} (See Theorem \ref{regular-basis}, Corollary \ref{crlreglasgow} and \ref{crlreglasgow1})  
\emph{\label{thm:regularity_normalizer_basis}
Let $(\mathcal{B} \subset \mathcal{A}, E)$ be an inclusion of finite-dimensional von Neumann algebras, where $E$ denotes the unique conditional expectation preserving the trace given in equation \ref{trace}.  
Then the following statements are equivalent:
\begin{itemize}
  \item[(i)]  There exists a unitary orthonormal basis of $\mathcal{A}$ over $\mathcal{B}$ contained in the normalizer
  $\mathcal{N}_{\mathcal{A}}(\mathcal{B})$;
  \item[(ii)] The inclusion $\mathcal{B} \subset \mathcal{A}$ is regular and the inclusion matrix satisfies the spectral condition.
\end{itemize}
Moreover, if either $\mathcal{B}$ or $\mathcal{A}$ is a matrix algebra, then the inclusion is regular if and only if it admits a unitary orthonormal basis contained in $\mathcal{N}_{\mathcal{A}}(\mathcal{B})$.}

\medskip

 Finally, we show that every regular inclusion has depth two, linking our results to (weak) Hopf-algebraic structures. More precisely, if the inclusion $\mathcal{B} \subset \mathcal{A}$ is regular, then the corresponding inclusion matrix has depth $2$ {in the sense of \cite{BKK}}.To see the connection between depth $2$ in subfactor theory and regularity, we refer to \cite{NV},\cite{BG}. Thus, our results may be viewed as a finite-dimensional analogue of the
structure theory of regular inclusions developed in the subfactor and
$C^*$-algebraic settings.

\medskip
\noindent\textbf{Organization of the paper.}
In Section~2 we recall preliminaries on finite-dimensional inclusions and unitary orthonormal bases.
Section~3 introduces normalizer matrices and establishes the classification of regular inclusions.
In Section~4 we study unitary orthonormal bases and show that regular inclusions necessarily have depth two.

\section{Preliminaries}
Throughout this article, we work with a triple $(\mathcal{B}\subseteq
\mathcal{A}, E)$ where $\mathcal{A}, \mathcal{B}$ are finite
dimensional von Neumann algebras with $\mathcal{B} $ being a unital subalgebra of $\mathcal{A}$ and $E:\mathcal{A}\to \mathcal{B}$ is
a conditional expectation map. Such a triple will be referred to, following \cite{BB}, as a subalgebra system.
\begin{dfn}
    Let $(\mathcal{B}\subseteq \mathcal{A}, E)$ be a subalgebra system. A family $\{W_j:0\leq
j\leq (d-1)\}$ (for some $d\in \mathbb{N}$) of elements of
$\mathcal{A}$ is called a \textit{ (right) Pimsner-Popa basis} for
$(\mathcal{B}\subseteq \mathcal{A}, E)$ if
\begin{equation*}\label{right basis} X= \sum _{j=0}^{d-1}W_jE(W_j^*X),
~~\forall ~X\in \mathcal{A}.\end{equation*} It is said to be an
\textit{ unitary orthonormal basis} if $W_j$ is a unitary for every
$j$, and
\begin{equation*}\label{orthonormality} E(W_j^*W_k)= \delta
_{jk},~~\forall ~0\leq j,k\leq (d-1).\end{equation*}
    
\end{dfn}

Throughout, we will take
\begin{equation}\label{decompositionA}
\mathcal{A}= M_{n_0}\oplus M_{n_1}\oplus \cdots \oplus M_{n_{s-1}};
\end{equation}
\begin{equation}\label{decompositionB}
\mathcal{B}= M_{m_0}\oplus M_{m_1}\oplus \cdots \oplus
M_{m_{r-1}}.\end{equation}
The associated dimension vectors of $\mathcal{A}$ and $\mathcal{B}$ are
$n'=\begin{bmatrix}
    n_0&n_1&\cdots&n_{s-1}
\end{bmatrix}^t$ and $m'=\begin{bmatrix}
    m_0&m_1&\cdots&m_{r-1}
\end{bmatrix}^t$ respectively.\\
\noindent
\textbf{Notation:} The block diagonal form $\begin{bmatrix}
    A_0&\\
    &\ddots\\
    &&A_n
\end{bmatrix}$ will be denoted by $\text{bl-diag}(A_1,\cdots, A_n)$.

\begin{dfn}   The {\em inclusion matrix\/}  of $\mathcal{B}$ in
$\mathcal{A}$ is given by an $s\times r$ matrix:
$$A=[a_{ij}]_{0\leq i\leq (s-1); 0\leq j\leq (r-1).} $$
where the algebra $M_{m_j}$  of $\mathcal{B}$ appears $a_{ij}$ times
in the algebra $M_{n_i}$ of $\mathcal{A}.$ Any inclusion of finite
dimensional von Neumann algebras $\mathcal{B}\subseteq \mathcal{A}$
is completely determined up to isomorphism by this triple  $(A, m', n')$ of one inclusion matrix and two dimension
vectors. Note that by dimension counting,
\begin{equation}
\sum _{j=0}^{r-1} a_{ij}m_j= n_i, ~~\forall ~0\leq i\leq (s-1)
\end{equation}
or in matrix  notation: $n'=Am'$.
\end{dfn}
So, throughout this article unless otherwise stated, we consider inclusions of the following form

\begin{align}
    i: \oplus_{j=0}^{r-1} M_{m_j} &\rightarrow \oplus_{i=0}^{s-1}M_{n_i}  \notag \\
    \oplus_{j=0}^{r-1}X_j &\longmapsto \oplus_{i=0}^{s-1} \text{bl-diag}(X_0\otimes\mathbf{1}_{a_{i0}},X_1\otimes\mathbf{1}_{a_{i1}},\cdots ,X_{r-1}\otimes\mathbf{1}_{a_{ir-1}} )\label{inclusion}
\end{align} 
corresponding to the inclusion matrix $A=[a_{ij}]_{0\leq i\leq (s-1); 0\leq j\leq (r-1)}$.

Let $Q_j$ be the projection onto the $j-$th summand in the decomposition $\mathcal{B}=\oplus_{j=0}^{r-1}M_{m_j}$ and similarly let $P_i$ be the projection onto the $i-$th summand in the decomposition $\mathcal{A}=\oplus_{i=0}^{s-1}M_{m_i}$ for $0\leq j \leq (r-1); 0\leq i\leq (s-1).$ Note that $Q_j, P_i$ are minimal central projections of $\mathcal{B} ~\text{and}~\mathcal{A},$ respectively.

\iffalse
\textbf{Notation:} Let $\{E^i_{jk}\}_{1\leq k,l\leq n_i}^{0\leq i\leq s-1}$ be matrix units of $\oplus_{i=0}^{s-1}M_{n_i}$ and $\{Q_j\}_{j=0}^{r-1}$  and $\{P_i\}_{i=0}^{s-1}$ be minimal central projections of $\oplus_{j=0}^{r-1} M_{m_j}$ and $\oplus_{i=0}^{s-1}M_{n_i}$ respectively. Then for each $0\leq l\leq r-1$, the central projections of $\mathcal{B}$ are given by,
\begin{equation}\label{minimal projection}
    Q_l=\oplus_{i=0}^{s-1}\sum\limits_{k=\sum\limits_{j=0}^{l-1} m_{j}a_{ij}+1}^{\sum\limits_{j=0}^{l}m_ja_{ij}}E^i_{kk}
\end{equation} with the convention that, for $l=0$, we take $\sum\limits_{j=0}^{l-1} m_{j}a_{ij}+1=0$ and $a_{i(-1)}=0$ 
\fi

\begin{dfn}
    Consider the inclusions $\mathcal{B}_1\subseteq \mathcal{A}_1$ and $\mathcal{B}_2\subseteq \mathcal{A}_2$ with inclusion matrix  $A_1$ and $A_2$ respectively. Then the inclusion matrix of the direct sum $\mathcal{B}_1\oplus \mathcal{B}_2\subseteq \mathcal{A}_1\oplus\mathcal{A}_2$ is given by bl-diag$(A_1,A_2)$
     and the inclusion matrix of the tensor product $\mathcal{B}_1\otimes \mathcal{B}_2\subseteq \mathcal{A}_1\otimes\mathcal{A}_2$ is given by $A_1\otimes A_2$.

\end{dfn}

\begin{dfn}
    A subalgebra system $(\mathcal{B}\subseteq \mathcal{A}, E)$ is said to have the U-property if it admits a unitary orthonormal basis.
\end{dfn}

\begin{thm}\label{BB1}\cite{BB}
    Let $(\mathcal{B}\subseteq \mathcal{A}, E)$ be a subalgebra system,
with
$Am'=n'$, where $A$ is the inclusion matrix, and
$m', n'$ are dimension vectors. Suppose $(\mathcal{B}\subseteq \mathcal{A}, E)$
admits a unitary orthonormal basis with $ d$-unitary elements.
 Then $A^tn'= dm'$.

\iffalse
\itemsep0em
 Consequently,
$$A^tAm'=dm'; ~~AA^tn'=dn'.$$
In particular, $m', n'$ are Perron-Frobenius
eigenvectors of $A^tA$ and $AA^t$ respectively with eigenvalue $d$.
\fi
\end{thm}

\begin{dfn}
    The normalizer of the inclusion $\mathcal{B}\subseteq \mathcal{A}$ is defined as \begin{equation*}
        \mathcal{N}_{\mathcal{A}}(\mathcal{B}):=\{U\in \mathcal{U}(\mathcal{A}):U\mathcal{B}U^*=\mathcal{B}\}.
    \end{equation*} where $\mathcal{U(A)}$ denotes the group of unitaries of $\mathcal{A}$.Equivalently, the normalizer can also be defined as \begin{equation*}
        \mathcal{N}_{\mathcal{A}}(\mathcal{B}):=\{U\in \mathcal{U}(\mathcal{A}):Ad_U\in Aut(\mathcal{B})\}.\\
        \end{equation*} where $Ad_U(x)=UxU^*.$ The inclusion $\mathcal{B}\subseteq \mathcal{A}$ is called regular if $\mathcal{A}=\text{span}\;\mathcal{N}_{\mathcal{A}}(\mathcal{B})$.
    
\end{dfn}
The following are the three important examples of regular inclusions that admit a unitary orthonormal basis in the normalizer:
\begin{enumerate}
    \item \textbf{Trivial Inclusion:} $(M_n(\mathbb{C}) \subset M_n(\mathbb{C}), E_0\;)$ where $E_0$  is the unique conditional expectation preserving the usual trace on $M_n(\mathbb{C})$\\
    In this case, $E_0=Id$, $\mathcal{N_A(B)}=\mathcal{U(A)}$ and the identity operator forms the basis.

    \item \textbf{Scalar Inclusion:} $(\mathbb{C} \subset \oplus_{i=0}^{s-1} M_{n_i}(\mathbb{C}), E_1)$ where $E_1$ is the unique conditional expectation preserving the state \begin{equation}\label{state}
\varphi (\oplus X_i)= \frac{1}{\sum _{i=0}^{s-1}n_i^2} \sum
_{i=0}^{s-1}n_i~\mbox {trace}~ (X_i), ~~X_i \in M_{n_i}, 0\leq i\leq
(s-1).\end{equation} \\
     In this inclusion, one can see that $\mathcal{N_A(B)}=\mathcal{U(A)}$ and its admits unitary orthonormal basis as proved in \cite{BB} Theorem 5.3. and \cite{CKP} Theorem 2.1.

    \item \textbf{Diagonal Inclusion:} $(\mathbb{C}^n \subset M_n(\mathbb{C}),E_2)$ where $E_2$  is the unique conditional expectation preserving the normalized trace on $M_n(\mathbb{C})$.\\
    Let $\{e_0,e_2,\cdots,e_{n-1}\}$ be the standard basis of $\mathbb{C}^n$ and let $U$ be the translation unitary  defined by $Ue_k=e_{k+1(mod\;n)}$. Then $\{U^k:k\in \mathbb{Z}_n\}$ is contained in the normalizer of $\mathbb{C}^n$  and form a unitary orthonormal basis. For further details, see \cite{CCKL} Example 2.2.
\end{enumerate}

\textbf{Tensor products and direct sums:}\\
Suppose $(\mathcal{B}_i \subseteq  \mathcal{A}_i, E_i)$ for $i=1,2$
are two finite dimensional inclusions with $U$-property. Suppose
$\{U_1(i) , \ldots , U_{d_i}(i)\}$ is a unitary o.n.b. in $\mathcal{N}_{\mathcal{A}_i}(\mathcal{B}_i)$
for
$(\mathcal{B}_i, \mathcal{A}_i, E_i)$ for $i=1,2.$ Take
$\mathcal{B}=\mathcal{B}_1\otimes \mathcal{B}_2$. Consider it as a
subalgebra of $\mathcal{A}=\mathcal{A}_1\otimes \mathcal{A}_2$  with the inclusion matrix being the tensor product
of inclusion matrices. Then it is seen easily that $E=E_1\otimes
E_2$ is a conditional expectation map from $\mathcal{A}$ to
$\mathcal{B}.$ For this inclusion, $(\mathcal{B}_1\otimes
\mathcal{B}_2\subseteq \mathcal{A}_1\otimes \mathcal{A}_2,
E_1\otimes E_2)$, we can observe that $\{ U_j(1)\otimes U_k(2):
0\leq j\leq d_1; 0\leq k\leq d_2\}$ is a unitary orthonormal basis in $\mathcal{N}_{\mathcal{A}_1\otimes \mathcal{A}_2}(\mathcal{B}_1 \otimes \mathcal{B}_2)$ as $\mathcal{N}_{\mathcal{A}_1}(\mathcal{B}_1)\otimes \mathcal{N}_{\mathcal{A}_2}(\mathcal{B}_2) \subseteq \mathcal{N}_{\mathcal{A}_1\otimes \mathcal{A}_2}(\mathcal{B}_1 \otimes \mathcal{B}_2)$. \\
Under the same setup, if $d_1=d_2$, the subalgebra system
$(\mathcal{B}_1\oplus \mathcal{B}_2\subseteq \mathcal{A}_1\oplus
\mathcal{A}_2, E_1\oplus E_2)$ has $U$-property. In fact, $\{
U_j(1)\oplus U_j(2): 0\leq j \leq d_1\}$ is a unitary orthonormal
basis in $\mathcal{N}_{\mathcal{A}_1\oplus \mathcal{A}_2}(\mathcal{B}_1 \oplus \mathcal{B}_2)$ as $\mathcal{N}_{\mathcal{A}_1}(\mathcal{B}_1)\oplus \mathcal{N}_{\mathcal{A}_2}(\mathcal{B}_2) \subseteq \mathcal{N}_{\mathcal{A}_1\oplus \mathcal{A}_2}(\mathcal{B}_1 \oplus \mathcal{B}_2)$. \\

\iffalse
In the next section, we will show that every regular inclusion of finite-dimensional $\mathbf{C}^*$-algebras can be obtained by forming suitable tensor products and direct sums of the inclusions described above.\newline
\fi

\section{Regular inclusions in finite dimension}
In this section, we develop a comprehensive framework for understanding regular inclusions of finite-dimensional von Neumann algebras. In Subsection 3.1, we introduce the notion of a normalizer matrix, a combinatorial invariant that encodes the structure of the normalizer entirely at the level of the inclusion matrix, and show that every normalizer matrix can be brought to a block-diagonal form reflecting the natural partitioning of row supports. Subsection 3.2 focuses on regular inclusions within matrix algebras, where we provide a detailed description of normalizers and establish that regularity imposes strict homogeneity conditions on the summands and equality of inclusion matrix entries. Finally, in Subsection 3.3, we extend these results to the general setting, proving that an inclusion is regular if and only if its inclusion matrix is a normalizer matrix and the dimensions of summands corresponding to each row support are equal. As a consequence, every regular inclusion decomposes into direct sums and tensor products of the fundamental building blocks identified in Section 2. This exposition lays the foundation for the subsequent analysis of unitary orthonormal bases.
\subsection{ Normalizer matrices}
To analyze the internal structure of regular finite-dimensional inclusions, we introduce the notion of a normalizer matrix. This matrix encodes  the combinatorial structure of the normalizer. Our main goal is to show that every normalizer matrix can be permuted into a block-diagonal form, where each diagonal block has rows with equal non-zero entries and all off-diagonal blocks are zero. This reflects the natural partition of row indices according to identical support patterns. The block-diagonal form provides a convenient framework for understanding the combinatorial structure of normalizers . As we will see, this concept will be used crucially in later sections to establish structural results, construct unitary orthonormal bases, and describe decompositions into fundamental building blocks.

Let $A = [a_{ij}]_{0 \le i \le (s-1),\, 0 \le j \le (r-1)}$ be an $s \times r$ matrix.  
For each $i \in \{0,1,\ldots,(s-1)\}$, we define the \textbf{support of the $i$-th row} of $A$ by
\[
Y_i := \{\, j \in \{0,1,\ldots,(r-1)\} \mid a_{ij} \neq 0 \,\}.
\]
We equip $Y_i$ with the \emph{natural order} inherited from $\{0,1,\ldots,r-1\}$; that is, if $Y_i = \{ x_1, x_2, \ldots, x_n \},$ then $x_1 < x_2 < \cdots < x_n.$

\begin{ppsn}\label{eqclass}
    Let $A$ and the row supports $Y_i$ be as defined above. Then the following statements are equivalent:
\begin{enumerate}
    \item For any two rows $i$ and $k$, if there exists a column $j$ such that $a_{ij} = 0$ and $a_{kj} \neq 0$, then for every column $l$, 
    \[
        a_{il} \neq 0 \;\; \Longrightarrow \;\; a_{kl} = 0.
    \]
    \item The row supports $\{Y_i\}_{i=0}^{s-1}$ form a partition of $\{0,1,\ldots,s-1\}$ under the equivalence relation
    \[
        x \sim y \quad \Longleftrightarrow \quad Y_x = Y_y.
    \]
\end{enumerate}\end{ppsn}

\begin{prf}
Suppose that the supports $Y_i$ and $Y_k$ are neither disjoint nor identical. Then there exist indices
\[
j \in Y_k \setminus Y_i \quad \text{(or } j \in Y_i \setminus Y_k \text{)} \quad \text{and} \quad l \in Y_i \cap Y_k.
\]
By the definition of row supports, this implies either $a_{ij} = 0$ and $a_{kj} \neq 0$, or $a_{ij} \neq 0$ and $a_{kj} = 0$, while both $a_{il}$ and $a_{kl}$ are nonzero. This contradicts the structural property required by condition (1). Therefore, for any $i \neq k$, the sets $Y_i$ and $Y_k$ are either disjoint or identical.

Conversely, assume that the family $\{Y_i\}$ forms a partition of $\{0,1,\dots,(s-1)\}$. Let $a_{ij} = 0$ and $a_{kj} \neq 0$ for some column $j$. Then $j \in Y_k \setminus Y_i$. If there exists a column $l$ such that $a_{il} \neq 0$ and $a_{kl} \neq 0$, then $l \in Y_i \cap Y_k$, which contradicts the assumption that the row supports form a partition. Hence, $a_{kl} = 0$ as required.
\qed 
\end{prf}

Recall that a matrix is called \emph{irredundant} if none of its rows or columns is zero.

\begin{dfn}\label{normalizer matrix}
    An irredundant matrix $A = [a_{ij}]$ is called a \textbf{normalizer matrix} if it satisfies the following conditions:
\begin{enumerate}
    \item In each row of $A$, all nonzero entries are equal.
    \item The row supports form a partition of the index set $\{0,1,\ldots,(s-1)\}$, where two indices $x$ and $y$ are equivalent if and only if
    \[
        x \sim y \quad \Longleftrightarrow \quad Y_x = Y_y .
    \]
\end{enumerate}\end{dfn}
We shall see later that this terminology reflects the fact that such matrices precisely encode the structure of the normalizer of the inclusion.

\begin{dfn}\cite{GHJ}
  Two $r\times s$ matrices $X_1$ and $X_2$ are defined to be {\em pseudo-equivalent} if there exist permutations matrices $P$ and $Q$ such that $X_2=PX_1Q$, namely appropriate exchanges of rows and columns convert $X_1$ to $X_2$.
\end{dfn}
The following proposition shows that pseudo-equivalent inclusion matrices give rise to isomorphic finite-dimensional inclusions. Although this fact is well known to experts, we include a proof for the convenience of the reader.
\begin{ppsn}\label{ppsn2}
   Let $\mathcal B \subset \mathcal A$ be a finite-dimensional inclusion with inclusion matrix $A$. 
If $B$ is pseudo-equivalent to $A$, say $B = PAQ$ for some permutation matrices $P$ and $Q$, then there exists an inclusion
\[
\widetilde{\mathcal B}
= \bigoplus_{j=0}^{r-1} M_{m_{\tau(j)}}
\;\subseteq\;
\bigoplus_{i=0}^{s-1} M_{n_{\sigma(i)}}
= \widetilde{\mathcal A},
\]
whose inclusion matrix is $B$, and such that
\[
(\mathcal B \subset \mathcal A) \cong (\widetilde{\mathcal B} \subset \widetilde{\mathcal A}).
\]\end{ppsn}
 
\begin{prf}
   Let $i \colon \mathcal{B} \longrightarrow \mathcal{A}$
denote the inclusion map corresponding to the inclusion matrix $A$. Explicitly, this inclusion is given by
\[
\bigoplus_{j=0}^{r-1} X_j
\;\longrightarrow\;
\bigoplus_{i=0}^{s-1}
\mathrm{bl\text{-}diag}
\bigl(
X_{\tau(0)} \otimes \mathbf{1}_{a_{i\tau(0)}},
\ldots,
X_{\tau(r-1)} \otimes \mathbf{1}_{a_{i\tau(r-1)}}
\bigr),
\]
which realizes the inclusion $\mathcal{B} \subset \mathcal{A}$ with inclusion matrix $A$.

Similarly, consider the inclusion $\tilde{i} \colon \widetilde{\mathcal{B}} \longrightarrow \widetilde{\mathcal{A}}$
corresponding to the inclusion matrix $PAQ$, given by
\[
\bigoplus_{j=0}^{r-1} X_{\tau(j)}
\;\longrightarrow\;
\bigoplus_{i=0}^{s-1}
\mathrm{bl\text{-}diag}
\bigl(
X_{\tau(0)} \otimes \mathbf{1}_{a_{\sigma(i)\tau(0)}},
\ldots,
X_{\tau(r-1)} \otimes \mathbf{1}_{a_{\sigma(i)\tau(r-1)}}
\bigr).
\]

Define a $*$-isomorphism $\theta \colon \mathcal{A} \longrightarrow \widetilde{\mathcal{A}}$ by
\[
\theta\!\left( \bigoplus_{i=0}^{s-1} Y_i \right)
=
\bigoplus_{i=0}^{s-1} Y_{\sigma(i)}.
\]
Clearly, $\theta$ is a $*$-isomorphism from $\mathcal{A}$ onto $\widetilde{\mathcal{A}}$.

To show that the inclusions are isomorphic, it suffices to verify that
\[
\theta\bigl(i(\mathcal{B})\bigr) = \tilde{i}(\widetilde{\mathcal{B}}).
\]
Indeed, we have
\begin{equation*}
\begin{aligned}
\theta\!\left(
\bigoplus_{i=0}^{s-1}
\mathrm{bl\text{-}diag}
\bigl(
X_{\tau(0)} \otimes \mathbf{1}_{a_{i\tau(0)}},
\ldots,
X_{\tau(r-1)} \otimes \mathbf{1}_{a_{i\tau(r-1)}}
\bigr)
\right)
=
\bigoplus_{i=0}^{s-1}
\mathrm{bl\text{-}diag}
&\bigl(
X_{\tau(0)} \otimes \mathbf{1}_{a_{\sigma(i)\tau(0)}},\ldots, \\
&X_{\tau(r-1)} \otimes \mathbf{1}_{a_{\sigma(i)\tau(r-1)}}
\bigr).
\end{aligned}
\end{equation*}
which coincides with $\tilde{i}(\widetilde{\mathcal{B}})$. Hence, the inclusions are isomorphic.
\qed
\end{prf}

\blmma\label{pseudo}
The class of normalizer matrices is preserved under pseudo-equivalence.
\elmma
\begin{prf}
Let $A$ be a {normalizer matrix}, and let $B$ be pseudo-equivalent to $A$. By definition, there exist permutation matrices $P$ and $Q$, corresponding to permutations $\sigma$ and $\tau$ of the row and column indices, respectively, such that $B = PAQ.$ Equivalently, if $A = [a_{ij}]$, then $B = [b_{ij}]$ with
\[
b_{ij} = a_{\sigma(i)\tau(j)}.
\]

We first show that $B$ is irredundant. Suppose, for contradiction, that the $i$-th row of $B$ is zero. Then
\[
b_{ik} = a_{\sigma(i)\tau(k)} = 0 \quad \text{for all } k,
\]
which implies that the $\sigma(i)$-th row of $A$ is zero, contradicting the irredundancy of $A$. Hence, no row of $B$ is zero. Similar argument shows that no column of $B$ is zero.

Next, we verify that in each row of $B$, all nonzero entries are equal. Indeed, if the nonzero entries in the $i$-th row of $B$ were not all equal, then the same would be true for the nonzero entries in the $\sigma(i)$-th row of $A$, contradicting the fact that $A$ is a normalizer matrix.

Finally, consider two rows $i$ and $k$ of $B$. Suppose there exists a column $j$ such that
\[
b_{ij} = a_{\sigma(i)\tau(j)} = 0
\quad \text{and} \quad
b_{kj} = a_{\sigma(k)\tau(j)} \neq 0.
\]
Then, by the defining property of a normalizer matrix, for every column $l$,
\[
a_{\sigma(i)\tau(l)} \neq 0 \;\Rightarrow\; a_{\sigma(k)\tau(l)} = 0.
\]
Equivalently,
\[
b_{il} \neq 0 \;\Rightarrow\; b_{kl} = 0.
\]

Thus $B$ satisfies all the defining properties of a normalizer matrix. Hence, the class of normalizer matrices is preserved under pseudo-equivalence.  \qed
\end{prf}

\begin{ppsn}\label{ppsn1}
Every normalizer matrix $A$ is pseudo-equivalent to a block-diagonal matrix
\[
\mathrm{bl\text{-}diag}(A_{11}, A_{22}, \ldots, A_{pp}),
\]
where each diagonal block $A_{ii}$ has the property that all entries in every row are equal and nonzero, and all off-diagonal blocks are zero.
\end{ppsn}
\begin{prf}
Let $A$ be a normalizer matrix, and let $\{R_k\}_{k=1}^{p}$ denote the partition of $\{0,1,\ldots,s-1\}$ induced by the equivalence relation on row indices described above.

For each $k \in \{1,2,\ldots,p\}$, set $s_k = |R_k|$ and let $r_k = |Y_x|$ for any $x \in R_k$. Write
\[
R_k = \{ k_1, k_2, \dots, k_{s_k} \}, \qquad
C_k = \{\,  k^1, k^2, \dots, k^{r_k} \,\} = Y_x \text{ for any } x \in R_k.
\]
        Define two permutations $\sigma$ and $\tau$ as follows:
\begin{equation}\label{sigma}
\sigma =
\begin{pmatrix}
0 & 1 & \cdots & s_1-1 & s_1 & \cdots & \sum_{i=1}^{2} s_i -1 & \cdots & \sum_{i=1}^{p} s_i & \cdots & s-1 \\
1_1 & 1_2 & \cdots & 1_{s_1} & 2_1 & \cdots & 2_{s_2} & \cdots & p_1 & \cdots & p_{s_p}
\end{pmatrix}
\end{equation}
\begin{equation}\label{tau}
\tau =
\begin{pmatrix}
0 & 1 & \cdots & r_1-1 & r_1 & \cdots & \sum_{i=1}^{2} r_i -1 & \cdots & \sum_{i=1}^{p} r_i & \cdots & r-1 \\
1^1 & 1^2 & \cdots & 1^{r_1} & 2^1 & \cdots & 2^{r_2} & \cdots & p^1 & \cdots &  p^{r_p}
\end{pmatrix}
\end{equation}

Let $P$ and $Q$ be the permutation matrices corresponding to $\sigma$ and $\tau$, respectively. Then define
\[
B := PAQ = [a_{\sigma(i)\tau(j)}]_{0 \le i \le s-1,\, 0 \le j \le r-1}.
\]
By Lemma \ref{pseudo}, since $A$ is an normalizer matrix and $B$ is pseudo-equivalent to $A$, the matrix $B$ is also an normalizer matrix. With respect to the chosen orderings, $B$ decomposes into block form 
$B=[A_{kl}]_{1\leq k,l\leq p}$where each block $$A_{kl}=\begin{pmatrix}
    a_{k_1l^1}&a_{k_1l^2}&\cdots&a_{k_1l^{r_l}}\\
    a_{k_2l^1}&a_{k_2l^2}&\cdots&a_{k_2l^{r_l}}\\
    \vdots&&\ddots&\vdots\\
    a_{k_{s_k}l^1}&a_{k_{s_k}l^2}&\cdots&a_{k_{s_k}l^{r_l}}
\end{pmatrix}$$ is a $s_k\times r_l$ matrix.\\

\textbf{Claim 1:} \textit{$A_{kl} = 0$ for $k \neq l$}.\\
Let $a_{k_x l^y}$ be an arbitrary entry of $A_{kl}$ with $k_x \in R_k$ and $ l^y \in C_l$. If $k \neq l$, then $l^y \notin C_k = Y_{k_x}$. Hence, $a_{k_x l^y} = 0$, proving the claim.

\medskip
\textbf{Claim 2:} \textit{For each $k$, every row of $A_{kk}$ consists of equal and non-zero entries.}\\
Fix $k$. For any $k_x \in R_k$ and $k^y \in C_k = Y_{k_x}$, we have $a_{k_x k^y} \neq 0$ by the definition of row support. Since $B$ is a normalizer matrix, all non-zero entries in a row are equal. Therefore, each row of $A_{kk}$ consists of equal non-zero entries.

\medskip
Combining Claims 1 and 2, we obtain
\[
PAQ = \mathrm{bl\text{-}diag}(A_{11}, A_{22}, \dots, A_{pp}),
\]
where each diagonal block $A_{ii} = [a_{i_x i^y}]_{1 \le x \le s_i,\, 1 \le y \le r_i}$ has rows with equal non-zero entries. This completes the proof. \qed
 \end{prf}

\subsection{Normalizers of Subalgebras of Matrix Algebras}
Throughout this subsection, a matrix algebra means a full matrix algebra. In this subsection, we examine regular inclusions where the ambient algebra is a full matrix algebra. Exploiting the block structure induced by the inclusion matrix, we characterize the normalizer and show that regularity imposes strong homogeneity constraints: the summands of the subalgebra must have equal dimensions, and all nonzero entries of the inclusion matrix must coincide. These observations provide a precise description of normalizers and set the stage for extending regularity results to the general finite-dimensional case in the next subsection.
\iffalse
Throughout, let
\[
B = \bigoplus_{j=0}^{r-1} M_{m_j}(\mathbb{C}) \subset M_n(\mathbb{C}) = A
\]
be an inclusion of finite-dimensional von Neumann algebras, with inclusion matrix $$A = \begin{bmatrix}
a_0 & a_1 & \cdots & a_{r-1}
\end{bmatrix},$$ so that $n = \sum_{j=0}^{r-1} m_j a_j .$
\fi
\medskip

We first show that regularity forces all entries of the inclusion matrix to coincide.

\begin{ppsn}\label{lmma1}
    If  $\mathbb{C}^r \subseteq M_n(\mathbb{C})$ is a regular inclusion, then all the entries of the inclusion matrix are equal.
\end{ppsn}

\begin{prf}
Let the inclusion matrix for inclusion $\mathcal B=\mathbb C^r\subset M_n(\mathbb C)$ be $\begin{bmatrix}
a_0&a_1&\cdots&a_{r-1}
\end{bmatrix}$ and let $Q_0,Q_1,\ldots,Q_{r-1}$ be the minimal central projections of $\mathcal B$, where
\[
a_j=\operatorname{rank}(Q_j),\qquad j=0,\ldots,r-1.
\]

Assume, toward a contradiction, that the entries of the inclusion matrix are not all equal. Reordering indices if necessary, we may assume
\[
a_0=\cdots=a_m<a_{m+1}\le\cdots\le a_{r-1}
\]
for some $m<r-1$.

For each $j$, let
\[
H_j:=\operatorname{Ran}(Q_j)\subset\mathbb C^n,
\]
so that
\[
\mathbb C^n=\bigoplus_{j=0}^{r-1} H_j.
\]
Define
\[
K:=H_0\oplus\cdots\oplus H_m.
\]

Let $U\in\mathcal N_{M_n(\mathbb C)}(\mathcal B)$. Since conjugation by $U$ leaves $\mathcal B$ invariant and the projections
$Q_0,\ldots,Q_{r-1}$ are precisely the minimal projections of $\mathcal B$, there exists a unique permutation $\sigma\in S_r$ such that
\[
UQ_jU^*=Q_{\sigma(j)} \qquad\text{for all } j.
\]
Taking ranges, we obtain
\[
U(H_j)=H_{\sigma(j)} \qquad\text{for all } j.
\]
As $U$ is unitary, it preserves dimensions, and hence
\[
\dim H_j=\dim H_{\sigma(j)} \qquad\text{for all } j.
\]

By construction, the subspaces $H_0,\ldots,H_m$ are exactly those of minimal dimension among the $H_j$. Since $U$ preserves dimensions and permutes the family $\{H_j\}$, it follows that
\[
\sigma(\{0,1,\ldots,m\})=\{0,1,\ldots,m\}.
\]
Consequently,
\[
U(K)=\bigoplus_{j=0}^m U(H_j)
=\bigoplus_{j=0}^m H_{\sigma(j)}
=K.
\]

Since $U$ is unitary and leaves $K$ invariant, it also leaves the orthogonal complement $K^\perp$ invariant. Indeed, for any $|f\rangle\in K^\perp$ and $|x\rangle\in K$,
\[
\langle Uf,x\rangle=\langle f,U^*x\rangle=0,
\]
because $U^*|x\rangle\in K$. Hence $U|f\rangle\in K^\perp$, and therefore
\[
U(K^\perp)=K^\perp.
\]

Choose unit vectors
\[
|f\rangle\in K^\perp, \qquad |e\rangle\in H_0\subset K,
\]
and define a rank-one operator $ |e\rangle\langle f|
\in M_n(\mathbb C)$.
Then $|e\rangle\langle f|(|f\rangle)=|e\rangle\neq0$, so $T$ maps a vector from $K^\perp$ into $K$.

On the other hand, every normalizing unitary $U$ satisfies $U(K^\perp)\subset K^\perp$, and the same is therefore true for any finite linear combination of normalizing unitaries. Consequently, no such operator can map a vector from $K^\perp$ into $K$. It follows that $T$ does not belong to the linear span of the normalizing unitaries.

Since the inclusion $\mathbb C^r\subset M_n(\mathbb C)$ is regular, the linear span of the normalizing unitaries equals $M_n(\mathbb C)$. The existence of the operator $T$ constructed above therefore contradicts regularity. Hence the assumption that the entries of the inclusion matrix are not all equal is false, and all entries of the inclusion matrix must be equal.\qed
\end{prf}

\begin{ppsn}\label{inducedreg}
    If inclusion $\mathcal{B}\subseteq \mathcal{A}$ is regular, then for every central projection $P\in \mathcal{Z(A)}$, the induced inclusion $P\mathcal{B}\subseteq P\mathcal{A}$ is also regular.
\end{ppsn} 
\begin{prf}
    For $U\in \mathcal{N}_{\mathcal{A}}(\mathcal{B})$, we have $PU\in \mathcal{N}_{P\mathcal{A}}(P\mathcal{B})$, which implies $P\mathcal{N}_{\mathcal{A}}(\mathcal{B}) \subseteq \mathcal{N}_{P\mathcal{A}}(P\mathcal{B})$. So we have
    \begin{equation*}
        P\mathcal{A}=\text{span} P\mathcal{N}_{\mathcal{A}}(\mathcal{B})\subseteq \text{span} \mathcal{N}_{P\mathcal{A}}(P\mathcal{B})
    \end{equation*} Thus, $P\mathcal{B}\subseteq P\mathcal{A}$ is regular. \qed
\end{prf}

\noindent
\textbf{Notation:} Consider the inclusion $\mathbb{C}^r\subset \oplus_{i=0}^{s-1}M_{n_i}$. We assume that its inclusion matrix has constant row structure, meaning that every nonzero entry in the $k$-th row is equal to $a_k$ and $Y_k$ denote the support of  $k$-th row. 

\iffalse 
Since the inclusion is unital, we have 
\[  \sum_{l\in Y_i}P_iQ_l=\mathbf{1}_{n_i}, \qquad 0\leq i\leq s-1.
\] where $\{Q_j\}_{j=0}^{r-1}$  and $\{P_i\}_{i=0}^{s-1}$ be minimal central projections of $\mathbb{C}^r$ and $\oplus_{i=0}^{s-1}M_{n_i}$ respectively.
Define \[
H_j^k:=\operatorname{Ran}(P_kQ_j)\subset\mathbb C^{n_i} \qquad  0\leq k\leq s-1 \;\; \text{and}\;\; j\in Y_k.
\]
Then we obtain a orthogonal decomposition
\[
\mathbb{C}^{n_i} \;=\; \bigoplus_{l\in Y_i}^{r-1} H_l^i,
\qquad
H_j^k:=\operatorname{Ran}(P_kQ_j)\subset\mathbb C^{n_i}
\]

 With respect to this decomposition, we identify $M_{n_i}(\mathbb{C})$ with an $r_i\times r_i$ block matrix algebra with blocks in $M_{a_i}(\mathbb{C})$.
Accordingly, the matrix algebra decomposes as 
\[
M_{n_i}(\mathbb{C}) = \bigoplus_{l,k\in Y_i}Q_l M_{n_i}(\mathbb{C}) Q_k,
\]
For $A_i\in M_{n_i}(\mathbb{C})$, we write $A_i^{(l,k)}:=Q_lA_iQ_k$ so that $A_i= \bigoplus_{l,k\in Y_i}Q_lA_iQ_k$ where each block $A_i^{(l,k)}\in M_{a_i}(\mathbb C)$
\fi

Since the inclusion $\mathcal B \subset \mathcal A$ is unital, we have
\[
\sum_{l \in Y_k} P_k Q_l = \mathbf 1_{n_k},
\qquad 0 \le k \le s-1,
\]
where $\{Q_j\}_{j=0}^{r-1}$ and $\{P_k\}_{k=0}^{s-1}$ denote the minimal central
projections of $\mathbb C^r$ and $\bigoplus_{k=0}^{s-1} M_{n_k}$, respectively.

For $k \in \{0,1,\ldots,s-1\}$ and $j \in Y_k$, define
\[
H_j^k := \operatorname{Ran}(P_k Q_j) \subset \mathbb C^{n_k}.
\]
Then
\[
\mathbb C^{n_k} = \bigoplus_{j \in Y_k} H_j^k
\]
is an orthogonal decomposition. With respect to this decomposition, the matrix
algebra $M_{n_k}(\mathbb C)$ is naturally identified with an
$|Y_k| \times |Y_k|$ block matrix algebra whose entries belong to
$M_{a_k}(\mathbb C)$.

More precisely, we obtain the decomposition
\[
M_{n_k}(\mathbb C)
=
\bigoplus_{l,j \in Y_k} Q_l \, M_{n_k}(\mathbb C) \, Q_j .
\]
For $A_k \in M_{n_k}(\mathbb C)$, we write
\[
A_k^{(l,j)} := Q_l A_k Q_j \in M_{a_k}(\mathbb C),
\]
so that
\[
A_k = \sum_{l,j \in Y_k} A_k^{(l,j)} .
\]
\color{black}

\medskip

\noindent We illustrate the above discussion with the following example.

\begin{xmpl}
Let $\mathcal B=\mathbb C^4 \subset \mathcal A := M_9(\mathbb C)\oplus M_4(\mathbb C)$
be an inclusion with inclusion matrix
\[
A=
\begin{bmatrix}
3 & 3 & 0 & 3\\
0 & 0 & 2 & 2
\end{bmatrix}.
\]
Let $\{P_0,P_1\}$ denote the minimal central projections of $\mathcal A$
corresponding to the summands $M_9(\mathbb C)$ and $M_4(\mathbb C)$,
and let $\{Q_j\}_{j=0}^3$ be the minimal central projections of $\mathcal B$.

For the first row, the nonzero entries occur at $j=0,1,3$, each equal to $3$.
Hence
\[
Y_0=\{0,1,3\}, \qquad a_0=3,
\]
and unitality gives
\[
P_0 Q_0 + P_0 Q_1 + P_0 Q_3 = P_0 = \mathbf 1_{9}.
\]
Accordingly, we obtain an orthogonal decomposition
\[
\mathbb C^{9}
=
H_0^0 \oplus H_1^0 \oplus H_3^0,
\qquad
H_j^0 := \operatorname{Ran}(P_0 Q_j),
\]
with $\dim H_j^0 = 3$ for $j\in Y_0$.

For the second row, the nonzero entries occur at $j=2,3$ each equal to $2$.
Thus
\[
Y_1=\{2,3\}, \qquad a_1=2,
\]
and
\[
P_1 Q_2 + P_1 Q_3 = P_1 = \mathbf 1_{4},
\]
yielding the decomposition
\[
\mathbb C^{4}
=
H_2^1 \oplus H_3^1,
\qquad
H_j^1 := \operatorname{Ran}(P_1 Q_j),
\]
with $\dim H_j^1 = 2$ for $j\in Y_1$.

Combining both summands, we obtain
\[
\mathbb C^{9}\oplus \mathbb C^{4}
=
(\mathbb C^3 \oplus \mathbb C^3 \oplus \mathbb C^3)
\;\oplus\;
(\mathbb C^2 \oplus \mathbb C^2).
\]

With respect to these decompositions, any element
$A=A_0\oplus A_1 \in M_9(\mathbb C)\oplus M_4(\mathbb C)$
admits the block-matrix form
\[
A
=
\begin{pmatrix}
A_0^{(0,0)} & A_0^{(0,1)} & A_0^{(0,3)}\\
A_0^{(1,0)} & A_0^{(1,1)} & A_0^{(1,3)}\\
A_0^{(3,0)} & A_0^{(3,1)} & A_0^{(3,3)}
\end{pmatrix}
\;\oplus\;
\begin{pmatrix}
A_1^{(2,2)} & A_1^{(2,3)}\\
A_1^{(3,2)} & A_1^{(3,3)}
\end{pmatrix},
\]
where, for $i=0,1$ and $l,k\in Y_i$,
\[
A_i^{(l,k)} := Q_l A_i Q_k \in
\begin{cases}
M_3(\mathbb C), & i=0,\\
M_2(\mathbb C), & i=1.
\end{cases}
\]

This example shows explicitly how the minimal central projections
$P_k$ and $Q_j$ determine the block structure of each simple summand of
$\mathcal A$ via the row supports $Y_k$ of the inclusion matrix.
\end{xmpl}

\color{black}

\iffalse

\textbf{Notation:} 
We fix integers $r,t \ge 1$ and write $n = rt$ and identify $M_n(\mathbb{C})$ with an $r \times r$ block matrix algebra with
blocks in $M_t(\mathbb{C})$, corresponding to the orthogonal decomposition
\[
\mathbb{C}^n \;=\; \bigoplus_{k=0}^{r-1} E_k,
\qquad
E_k := \operatorname{span}\{ e_{kt+1}, \ldots, e_{(k+1)t} \}.
\]
Let $Q_0,\ldots,Q_{r-1}$ denote the orthogonal projections onto the subspaces
$E_0,\ldots,E_{r-1}$, respectively. For $A \in M_n(\mathbb{C})$ and
$k,l \in \{0,1,\ldots,r-1\}$, we denote by $A^{(l,k)}$ the $(l,k)$-th block of
$A$, defined by
\[
A^{(l,k)} := Q_l A Q_k \in M_t(\mathbb{C}).
\]

To justify this identification, observe that with respect to the standard
basis of $\mathbb{C}^n$, the projection $Q_k$ is the diagonal matrix selecting
the basis vectors $e_{kt+1},\ldots,e_{(k+1)t}$. Consequently, the operator
$Q_l A Q_k$ describes the action of $A$ from $E_k$ into $E_l$, and its matrix
entries are precisely the entries $(A)_{\alpha\beta}$ with
\[
lt+1 \le \alpha \le (l+1)t,
\qquad
kt+1 \le \beta \le (k+1)t.
\]

\fi

%For any $A \in M_n$ and $k,l\in\{0,1,\cdots,r-1\}$, we denote by $A^{(l,k)}$  the $(l,k)$-th block of $A$ . Explicitly, the $(l,k)$th block consists of the matrix entries $(A)_{\alpha \beta}$ with $lt +1\leq \alpha \leq (l+1)t$ and $kt+1\leq \beta \leq (k+1)t$.\\

\noindent In the following we provide the structure of $U\in \mathcal{N}_{(\mathcal{A})}(\mathbb{C}^r)$ with a specific condition on inclusion matrix.

\begin{ppsn}\label{ppsn:block_structure}
Let $\mathcal{B} = \mathbb{C}^r \subset \oplus_{i=0}^{s-1}M_{n_i}(\mathbb{C})$ be an inclusion whose
inclusion matrix having the property that, in each row, all the non-zero entries are equal. Let this common non-zero entry in the $k$-th row be $a_k$. Then by dimension counting, we have
$n_k = r_ka_k$, where $r_k=|Y_k|$. Identify $M_{n_k}(\mathbb{C})$ with an $r_k \times r_k$ block matrix algebra
with blocks in $M_{a_k}(\mathbb{C})$.

Then for every unitary $U \in \mathcal{N}_{\mathcal{A}}(\mathcal{B})$ there exists a \emph{unique}
bijective map $\sigma_i$ on $Y_i$, for each $i\in \{0,1,\cdots, s-1\}$ such that
\[
U_iP_ i Q_k U_i^* =P_i Q_{\sigma_i(k)} \qquad \text{for all } k \in Y_i.
\]
Equivalently,
\[
U_i^{(l,k)} \neq 0 \quad \Longleftrightarrow \quad l = \sigma_i(k).
\] where $U_i^{(l,k)}:=P_iQ_lU_iQ_k \in M_{a_i}(\mathbb{C})$ for $l,k\in Y_i$.
Moreover, whenever $U_i^{(l,k)} \neq 0$, one has $U_i^{(l,k)} \in \mathcal{U}(M_{a_i}(\mathbb{C}))$.
In particular, every $U_i$ is a block permutation matrix with unitary
blocks.
\end{ppsn}

\begin{xmpl}
    Let $\mathcal B=\mathbb C^4\subset M_9(\mathbb C)\oplus M_4(\mathbb C)$ with inclusion matrix $\begin{bmatrix}
    3&3&0&3\\
    0&0&2&2
\end{bmatrix}$.
Let $U\in \mathcal{N}_{(\mathcal{A})}(\mathbb{C}^r)$
Suppose $U_0$ corresponds to the map $\sigma_0$ where $\sigma_0(0)=1$, $\sigma_0(1)=3$ and $\sigma_0(3)=0$. Then
Proposition~\ref{ppsn:block_structure} implies that
\[
U_0^{(i,j)} \neq 0 \quad \Longleftrightarrow \quad i = \sigma_0(j),\;\; \text{for}\;\; i,j\in Y_0
\]
so that $U_0$ has exactly three nonzero blocks:
\[
U_0^{(0,3)}, \quad U_0^{(1,0)}, \quad U_0^{(3,1)} \in \mathcal{U}(3),
\]
and all other blocks are zero. Explicitly, the matrix looks like
\[
U_0 =
\begin{pmatrix}
0 & 0 & U_0^{(0,3)} \\
U_0^{(1,0)} & 0 & 0 \\
0 & U_0^{(3,1)} & 0
\end{pmatrix}.
\]
In this case, the associated permutation is uniquely determined by
$U_0Q_jU_0^* = Q_{\sigma_0(j)}$ for $j\in Y_0$.
 Similarly, if $U_1$ corresponds to the map $\sigma_1$ such that $\sigma_1(2)=3$ and $\sigma_1(3)=2$.
Then, \[
U_1 =
\begin{pmatrix}
0 & U_1^{(2,3)} \\
U_1^{(3,2)} & 0 \\
\end{pmatrix},
\]
where \[
U_1^{(2,3)}, \quad U_1^{(3,2)}\in \mathcal{U}(2).
\]

\end{xmpl}

\noindent \textbf{Proof of Proposition \ref{ppsn:block_structure}:}
Let $U \in \mathcal{N}_{\mathcal{A}}(\mathcal{B})$ be a unitary normalizer then $U_i:=P_iU\in \mathcal{N}_{P_i\mathcal{A}}(P_i\mathcal{B}) $.  

Since each $P_iQ_j$ for $j\in Y_i$ is a minimal projection in the commutative algebra $P_i\mathcal{B}$, conjugation by $U_i$ maps $P_iQ_j$ to another minimal projection $P_iQ_k$ for some $k\in Y_i$  in $P_i\mathcal{B}$.  
Hence, for each $j\in Y_i$, there exists a unique bijective map $\sigma_i$ on $Y_i$ such that
\[
U_iP_i Q_j U_i^* =P_iQ_{\sigma_i(j)}.
\]
It follows that $\sigma_i$ defines a bijective map on $Y_i$ because $U_i$ is unitary in $M_{n_i}$ and thus bijective on projections corresponding to $Y_i$. This establishes the existence and uniqueness of the maps $\sigma_i$ associated to $U$.

Now write $U_i$ in the block form with respect to the decomposition
\[
M_{n_i}(\mathbb{C}) = \bigoplus_{u,v\in Y_i}Q_u M_{n_i}(\mathbb{C}) Q_v,
\]
and recall for $l,k\in Y_i$
\[
U_i^{(l,k)} := Q_l U_iQ_k \in M_{a_i}(\mathbb{C}). 
\]
 Then clearly
\[
U_i Q_k = \sum_{l\in Y_i} U_i^{(l,k)}, \qquad Q_l U_i = \sum_{k\in Y_i}U_i^{(l,k)}.
\]

Since $U_iP_i Q_k U_i^* = P_iQ_{\sigma_i(k)}$, we can rewrite this as
\[
U_i Q_k = Q_{\sigma_i(k)} U_i = \sum_{l\in Y_i} Q_{\sigma_i(k)} U_i Q_l = \sum_{l\in Y_i} U^{(\sigma_i(k),l)}.
\]
\iffalse
In particular, every element of $M_{n_i}(\mathbb{C})$ admits a unique block
decomposition with respect to these corners.
Writing
\[
U_i Q_k = \sum_{l\in Y_i} Q_l U_i Q_k
\quad \text{and} \quad
Q_{\sigma_i(k)} U_i = \sum_{l\in Y_i} Q_{\sigma_i(k)} U_i Q_l,
\]
and 
\fi

Using the equality $U_iQ_k = Q_{\sigma_i(k)}U_i$, the uniqueness of the block
decomposition forces
\[
Q_l U_i Q_k = 0 \quad \text{for all } l \neq \sigma_i(k).
\]

In other words, all blocks $U_i^{(l,k)}$ vanish unless $l = \sigma_i(k)$:
\[
U_i^{(l,k)} = 0 \quad \text{for all } l \neq \sigma_i(k).
\]
Thus, for each $k\in Y_i$, exactly one block row contains nonzero entries, corresponding to the map $\sigma_i$. Equivalently,
\[
U_i^{(l,k)} \neq 0 \quad \Longleftrightarrow \quad l = \sigma_i(k).
\]

Conversely, suppose $U_i$ is a unitary whose block decomposition satisfies
\[
U_i^{(l,k)} \neq 0 \quad \text{iff} \quad l = \sigma_i(k)
\]
for some bijective map $\sigma_i$ on $Y_i$, and each nonzero $U_i^{(l,k)}$ is itself unitary in $M_{a_i}(\mathbb{C})$. Then for each $k\in Y_i$,
\[
U_i Q_k = U_i^{(\sigma_i(k),k)}, \quad Q_{\sigma_i(k)} U_i = U^{(\sigma_i(k),k)},
\]
so $U_i Q_k = Q_{\sigma_i(k)} U_i$, which implies
\[
U_i Q_k U_i^* = Q_{\sigma_i(k)}.
\]
Hence any unitary with this block structure normalizes $P_i\mathcal{B}$, showing the converse.

Finally, we show that each nonzero block $U_i^{(l,k)}$ is unitary.
By the block-support property established above, for each fixed $l$ there exists
a unique index $k$ such that $U_i^{(l,k)} \neq 0$, and similarly for each fixed $k$
there exists a unique $l$ with this property.

Since $U_i$ is unitary in $M_{n_i}$, we have $U_iU_i^*=\mathbf{1}_{n_i}$.  For $l\in Y_i$, this yields
\[
(U_iU_i^*)^{(l,l)}=\sum_{j\in Y_i} U_i^{(l,j)}(U_i^{(l,j)})^*=\mathbf{1}_{a_i}.
\]
All summands vanish except for $j=k$, and hence
\[
U_i^{(l,k)}(U_i^{(l,k)})^*=\mathbf{1}_{a_i}.
\]

Similarly, from $U_i^*U_i=\mathbf{1}_{n_i}$ and taking $k\in Y_i$ we obtain
\[
(U_i^*U_i)^{(k,k)}=\sum_{l\in Y_k}(U_i^{(l,k)})^*U_i^{(l,k)}=\mathbf{1}_{a_i},
\]
which again reduces to
\[
(U_i^{(l,k)})^*U_i^{(l,k)}=\mathbf{1}_{a_i}.
\]

Therefore, $U_i^{(l,k)}$ is a unitary element of $M_{a_i}(\mathbb{C})$.

Combining all observations, we conclude that every $P_iU$ for $U\in \mathcal{N}_{\mathcal{A}}(\mathbb{C}^r)$  is a block-permutation matrix with unitary blocks and the proof is complete.\qed

\color{black}

We recall the standard notion of a homogeneous finite-dimensional $C^*$-algebra.
\begin{dfn}
A finite-dimensional von Neumann algebra $\mathcal A$ is called \emph{homogeneous} if there exist integers $r,n \ge 1$ such that
\[
\mathcal A \cong \bigoplus_{i=1}^{r} M_n(\mathbb C).
\]
Equivalently, $\mathcal A$ is a finite direct sum of matrix algebras, all of the same dimension.
\end{dfn}

\iffalse
For the remainder of this section, we consider a finite-dimensional inclusions $\mathcal B \subset \mathcal A$ be a finite-dimensional inclusion as in
\eqref{inclusion}, with inclusion matrix $A_{\mathcal B}^{\mathcal A}=[a_{ij}]$.
The inclusion $i:\mathcal B\hookrightarrow \mathcal A$ associated with
$A_{\mathcal{B}}^{\mathcal{A}}=[a_{ij}]$ is defined componentwise by
\begin{equation}
i(X)_i
=
\bigoplus_{j=0}^{r-1} \left(X_j\otimes \mathbf 1_{a_{ij}}\right),
\qquad
X=\bigoplus_{j=0}^{r-1} X_j\in\mathcal B.
\end{equation}
\fi
\begin{ppsn}
Let $\mathcal B \subset \mathcal A$ be a finite-dimensional inclusion as in
\eqref{inclusion}, with inclusion matrix $A_{\mathcal B}^{\mathcal A} = [a_{ij}]_{0\le i\le (s-1),\;0\le j\le (r-1)}.$
Then the inclusion matrix of the restricted inclusion $\mathcal Z(\mathcal B) \subset \mathcal A$
is given by
\[
A_{\mathcal Z(\mathcal B)}^{\mathcal A}
=
\bigl[\, m_j\, a_{ij} \,\bigr]_{0\le i\le (s-1),\;0\le j\le (r-1)}.
\]
\end{ppsn}

\begin{prf}
    The centre of $\mathcal{B}$, denoted by $
\mathcal{Z(B)}$, equals $\bigoplus_{j=0}^{r-1}\mathbb{C}\,\mathbf{1}_{m_j}\cong \mathbb{C}^r .
$ For the inclusion $\mathcal{Z(B)}\subset \mathcal{A}$, the embedding is given by  \begin{align*}
    i:\ \bigoplus_{j=0}^{r-1}\mathbb{C}\,\mathbf{1}_{m_j}\ &\longrightarrow\ \bigoplus_{i=0}^{s-1}M_{n_i}\\
    \bigoplus_{j=0}^{r-1}\alpha_j\,\mathbf{1}_{m_j}\ &\longmapsto\ 
\bigoplus_{i=0}^{s-1}\; \text{bl-diag}(\alpha_0\,\mathbf{1}_{m_0 a_{i0}},\alpha_1\,\mathbf{1}_{m_1 a_{i1}},\cdots,\alpha_{r-1}\,\mathbf{1}_{m_{r-1} a_{i(r-1)}}).
\end{align*}

Thus, the inclusion matrix for $\mathcal{Z(B)}\subset \mathcal{A}$,
corresponding to the inclusion $\mathcal{B}\subset \mathcal{A}$, is  
$$A_{\mathcal{Z(B)}}^{\mathcal{A}}=[\,m_j a_{ij}\,].$$ \qed

\end{prf}

\begin{ppsn} \label{ppsn: regular homogeneous}
    If the inclusion $\mathcal{B}\subseteq M_n(\mathbb{C})$ is regular, then $\mathcal{B}$ is homogeneous and the entries of the inclusion matrix are equal. 
\end{ppsn}

\begin{prf}
We proceed in several steps.

\medskip
\noindent
\textbf{Step 1: Normalizers preserve the center.}  
We claim that
\[
\mathcal N_{M_n}(\mathcal B) \subseteq \mathcal N_{M_n}(\mathcal Z(\mathcal B)).
\]  
Indeed, if $U \in \mathcal N_{M_n}(\mathcal B)$, then $\mathrm{Ad}_U \in \mathrm{Aut}(\mathcal B)$.  
Since any automorphism preserves the center, we have $\mathrm{Ad}_U \in \mathrm{Aut}(\mathcal Z(\mathcal B))$, hence $U \in \mathcal N_{M_n}(\mathcal Z(\mathcal B))$.

As a consequence, if $\mathcal B \subseteq M_n(\mathbb C)$ is regular, then the inclusion $\mathcal Z(\mathcal B) \subseteq M_n(\mathbb C)$ is also regular.

\medskip
\noindent
\textbf{Step 2: Regularity of the center forces equality of scalar multiplicities.}  
Let $\mathcal Z(\mathcal B) \cong \bigoplus_{j=0}^{r-1} \mathbb C \mathbf 1_{m_j}$, and  its inclusion matrix in $M_n(\mathbb C)$ is $[m_0 a_0 \;\; m_1 a_1 \;\; \cdots \;\; m_{r-1} a_{r-1}].$
By Proposition~\ref{lmma1}, regularity implies 
\[
m_0 a_0 = m_1 a_1 = \cdots = m_{r-1} a_{r-1}.
\]  
Hence, by Proposition~\ref{ppsn:block_structure}, each $U \in \mathcal N_{M_n}(\mathcal B)$ is a block permutation matrix with unitary blocks.

\medskip
\noindent
\textbf{Step 3: Comparison of block sizes.}  
We claim that if $m_i < m_j$, then for any $U \in \mathcal N_{M_n}(\mathcal B)$ the block $U^{(i,j)} = 0$.

Indeed, consider an element
\[
X = \operatorname{diag}(X_0 \otimes \mathbf 1_{a_0}, \dots, X_{r-1} \otimes \mathbf 1_{a_{r-1}}) \in \mathcal B,
\]
with $X_i \in M_{m_i}(\mathbb C)$.  
Suppose $U^{(i,j)} \neq 0$. Then the $(i,i)$-block of $UXU^*$ is
\begin{align*}
(UXU^*)^{(i,i)}
&= \sum_{k=0}^{r-1} (UX)^{(i,k)} (U^*)^{(k,i)} \\
&= (UX)^{(i,j)} (U^*)^{(j,i)} \\
&= U^{(i,j)} (X_j \otimes \mathbf 1_{a_j}) U^{*(j,i)} \in M_{m_i} \otimes \mathbf 1_{a_i}.
\end{align*}

Thus, we would have
\[
U^{(i,j)} (M_{m_j} \otimes \mathbf 1_{a_j}) U^{*(j,i)} \subseteq M_{m_i} \otimes \mathbf 1_{a_i}.
\]  
Comparing dimensions, this requires $m_j \le m_i$, which contradicts $m_i < m_j$. Therefore $U^{(i,j)} = 0$.

However, if this were the case for some $i \neq j$, then every linear combination of normalizing unitaries would have zero in the corresponding $(i,j)^{th}$-block, contradicting the regularity of $\mathcal B \subseteq M_n(\mathbb C)$.

\medskip
\noindent
\textbf{Step 4: Homogeneity and equality of inclusion matrix entries.}  

From Step 2, we have $m_0 a_0 = m_1 a_1 = \cdots = m_{r-1} a_{r-1}$.  
Step 3 showed that if $m_i < m_j$, then the $(i,j)^{th}$-block of any normalizing unitary must vanish, contradicting regularity.  

By symmetry, the same argument rules out $m_j < m_i$.  
Hence, we conclude that
\[
m_0 = m_1 = \dots = m_{r-1}.
\]

Since the products $m_i a_i$ are equal, it follows that all $a_i$ are also equal.  
Consequently, $\mathcal B$ is homogeneous, and all entries of the inclusion matrix are equal. \qed
\end{prf}

\color{black}

%   \begin{rmrk}
%    Note that the converse of the above statement is also true i.e., if $B$ is homogeneous then $B=\oplus_{j=1}^{r-1}M_l(\mathbb{C})$, then from the definition of inclusion matrix we have $a_0rl=n$.
%  \begin{equation*}
%     M_n(\mathbb{C}) = M_k(\mathbb{C}) \otimes M_l(\mathbb{C}) \otimes M_{a_0}(\mathbb{C})
%  \end{equation*}

 %   \begin{equation*}
 %       B= \mathbb{C}^k\otimes M_l(\mathbb{C}) \otimes \mathbb{C}
%    \end{equation*}
% since each inclusion $\mathbb{C}^k\subset M_k(\mathbb{C})$, $\mathbb{C}\subset M_{a_0}(\mathbb{C})$  and $M_l(\mathbb{C})\subset M_l(\mathbb{C})$ are regular so is their tensor product.
%\end{rmrk}

\subsection{Regularity of Inclusions in the General Case}
We now extend the analysis of regular inclusions from the special case of matrix algebras to arbitrary finite-dimensional von Neumann algebras. Building on the combinatorial framework of normalizer matrices introduced in Subsection 3.1 and the structural results for matrix algebra inclusions in Subsection 3.2, we show that regularity in the general setting is entirely captured by the inclusion matrix and the dimensions of the simple summands. The characterization allows us to decompose any regular inclusion, up to isomorphism, into direct sums and tensor products of the canonical building blocks identified earlier, thereby providing a complete and unified description of regularity in finite dimensions.

\blmma\label{equal dimension}
    If the inclusion $\mathcal{B} \subset \mathcal{A}$ is regular, then for each row of the inclusion matrix  all the non-zero entries are equal. Moreover, for each index $j\in Y_i$, the summands $M_{m_j}$'s have the same dimension, where $i\in \{0,1,\cdots, s-1\}$.
\elmma
\begin{prf}
Fix $k\in \{0,1,\cdots,s-1\}$ and let the support of $k^{th}$ row be $Y_k=\{x_1,x_2,\cdots,x_{r_k}\}$. By equation \ref{inclusion} and  Lemma \ref{inducedreg}, we have \[
P_k\mathcal{B}=\text{bl-diag}(M_{m_{x_1}}\otimes \mathbf{1}_{a_{kx_{1}}},\cdots,M_{m_{x_{r_k}}\otimes \mathbf{1}_{a_{kx_{r_k}}}})\subset M_{n_k}=P_k\mathcal{A}
\]is regular. Equivalently, the inclusion
\[
\oplus_{j\in Y_i}M_{m_j}\subset M_{n_k}
\]
with inclusion matrix $[a_{kj}]_{j\in Y_i}$, is 
    regular. By proposition \ref{ppsn: regular homogeneous}, it follows that  $P_k\mathcal{B}$ is homogenoeus i.e., $M_{m_j}$'s have same dimension for all $j\in Y_k$ and the entries of the above inclusion matrix are non-zero and equal which implies that all the non-zero entries in each row of the inclusion matrix are equal. \qed
\end{prf}

\begin{xmpl}
    Consider the inclusion $\oplus_{j=0}^{3}M_{m_j}\subset \oplus_{i=0}^{1}M_{n_i}$
 with inclusion matrix $A=\begin{bmatrix}
     a_{00}&0&a_{02}&0\\
     0&a_{11}&0&a_{13}
 \end{bmatrix}$. Here we have $Y_0=\{0,2\}$ and $Y_1=\{1,3\}$.
Then by above proposition we have $m_0=m_2$ and $m_1=m_3$.

\end{xmpl}

%  Let $Y_i:=\{j\in\{0,1,\cdots,r-1\} : a_{ij}\neq 0\}$ and $X\subset \{0,1,\cdots ,s-1\}$ then the inclusion matrix for the inclusion $(\sum\limits_{i\in X}P_i)\mathcal{B}\subset (\sum\limits_{i\in X}P_i)\mathcal{A}$ is $[a_{ij}]_{i\in X,\;j\in \cup_{i\in X}Y_i}$, which we call as induced inclusion matrix.

\iffalse
In this case, let us represent elements of  $P_i\mathcal{A}=M_{n_i}$ as $r_i\times r_i$  block matrix, where $r_i=|Y_i|$ and each block has size $\frac{n_i}{r_i}\times\frac{n_i}{r_i}$. We label the blocks of  
$U_i\in M_{n_i}$ as  $U_i^{(k,l)}$ for $k,l\in Y_i$ preserving the natural order of indices.\\
As an illustration, consider the inclusion $$\oplus_{j=0}^{2}M_{m_j}\subset \oplus_{i=0}^1M_{n_i}$$ with inclusion matrix $A=\begin{bmatrix}
    a&0&a\\
    0&b&b
\end{bmatrix}$. 

In this case, $Y_0=\{0,2\}$ and $Y_1=\{1,2\}$. Accordingly, the 
block labeling of $T\in M_{n_0}$ is $$\begin{bmatrix}
T^{(0,0)}&T^{(0,2)}\\
T^{(2,0)}&T^{(2,2)}
\end{bmatrix}$$ where each block is of size $\frac{n_0}{2}\times \frac{n_0}{2}$     and the block labeling of $S\in M_{n_1}$ is $$\begin{bmatrix}
   S^{(1,1)}&S^{(1,2)}\\
S^{(2,1)}& S^{(2,2)} 
\end{bmatrix}$$ where each block is of size $\frac{n_1}{2}\times \frac{n_1}{2}$. \\

With this block notation, we have the following description, (similar to Proposition \ref{ppsn:block_structure}) for $(k,l)\in Y_i$ and $U\in \mathcal{N_A(B)}$. \begin{equation*}
   (P_iU= U_i)^{(l,k)}=\begin{cases}
        \neq0 & \text{iff} \:\; U_iP_iQ_kU_i^*=P_iQ_l,\\
        0& \text{Otherwise}.
    \end{cases}
\end{equation*} 
\fi

\begin{thm} \label{mainthm}
The inclusion $\mathcal{B} \subset \mathcal{A}$ is regular if and only if
\begin{enumerate}\itemsep0em
        \item The inclusion matrix is an \textit{normalizer matrix}.
        \item For each $i \in \{0,1,\ldots,s-1\}$, all summands $M_{m_j}$ with $j \in Y_i$ have the same dimension.
    \end{enumerate}

As a consequence, any regular inclusion is isomorphic to tensor product and direct sum of inclusions $M_k(\mathbb{C}) \subset M_k(\mathbb{C}),$  $\mathbb{C} \subset \oplus_{i=0}^{v-1} M_{l_i}(\mathbb{C})$ and $\mathbb{C}^t \subset M_t(\mathbb{C})$.
\end{thm} 
\begin{proof}
Assume that the inclusion $\mathcal B\subset\mathcal A$ is regular.
By Lemma~\ref{equal dimension}, all nonzero entries in each row of the inclusion
matrix coincide, and for each $i\in\{0,1,\dots,s-1\}$ all summands
$M_{m_j}$ with $j\in Y_i$ have the same dimension.  
It  remains to show that the row supports of the inclusion matrix form
equivalence classes, i.e.\ that the matrix is a normalizer matrix.

Equivalently, we prove the following claim.

\medskip
\noindent\textbf{Claim.}
\emph{If $a_{ij}=0$ and $a_{kj}\neq0$, then $a_{il}\neq0$ implies $a_{kl}=0$.}

\medskip
Let $P=P_i+P_k$ be the corresponding central projection in $\mathcal A$.
By Proposition~\ref{inducedreg}, the induced inclusion
$P\mathcal B\subset P\mathcal A$ is regular, and moreover $P\mathcal N_{\mathcal A}(\mathcal B)
\subset \mathcal N_{P\mathcal A}(P\mathcal B)).$\\
Suppose $a_{kl}\neq0$.  
We claim that for every $U\in\mathcal N_{\mathcal A}(\mathcal B) \subset \mathcal N_{\mathcal A}(\mathcal Z(B))$,
\[
(P_kU)^{(l,j)} = U_k^{(l,j)} = 0.
\]
This block entry is well defined since $a_{kj}\neq0$ and $a_{kl}\neq0$ imply
$j,l\in Y_k$.

\medskip
Let $U\in\mathcal N_{\mathcal A}(\mathcal B)$ and decompose
$PU=U_i+U_k$ according to the central decomposition of $P\mathcal A$.
Then
\[
PUQ_jU^*
= U_iP_iQ_jU_i^* + U_kP_kQ_jU_k^*
= 0 \oplus U_kP_kQ_jU_k^*,
\]
since $a_{ij}=0$ implies $P_iQ_j=0$.

\medskip
Assume, in a contradiction, that $U_k^{(l,j)}\neq0$.
By Proposition~\ref{ppsn:block_structure}, this implies
\[
0 \oplus U_kP_kQ_jU_k^* = 0 \oplus P_kQ_l.
\]
Since the $Q_j$ are minimal central projections of $\mathcal B$,
there exists another minimal projection $Q_x$ such that $UQ_jU^* = Q_x.$
Hence,
\[
PUQ_jU^* = PQ_x = P_iQ_x \oplus P_kQ_x.
\]
Comparing with the previous expression, we obtain
\[
0 \oplus P_kQ_l = P_iQ_x \oplus P_kQ_x.
\]
Equating the second summand yields $x=l$, and therefore
\[
0 \oplus P_kQ_l = P_iQ_l \oplus P_kQ_l.
\]
However, $P_iQ_l\neq0$ since $a_{il}\neq0$, which is a contradiction.
Thus $U_k^{(l,j)}=0$.

\medskip
This shows that for all
$A\in\operatorname{span}(\mathcal N_{\mathcal A}(\mathcal B))$,
the block entry $A_k^{(l,j)}$ vanishes.
This contradicts the regularity of
$\mathcal B \subset\mathcal A$ unless $a_{kl}=0$.
The claim follows, and hence the inclusion matrix is a normalizer matrix.

\medskip

Assume now that the inclusion matrix is a normalizer matrix and that
all $M_{m_j}$ with $j\in Y_i$ have equal dimension.
By Propositions~\ref{ppsn1} and~\ref{ppsn2}, the inclusion
$\mathcal B\subset\mathcal A$ is isomorphic to
\[
\bigoplus_{j=0}^{r-1} M_{m_{\tau(j)}}
\;\subseteq\;
\bigoplus_{i=0}^{s-1} M_{n_{\sigma(i)}},
\]
with block-diagonal inclusion matrix
\[
A'=\operatorname{bl\text{-}diag}(A_{11},A_{22},\dots,A_{pp}),
\]
where each block $A_{kk}=[a_{k_i k^j}]_{1\le i\le s_k,\,1\le j\le r_k}$.

Since $m_j=m_k$ for all $j\in Y_k$, we may rewrite
\[
\bigoplus_{j=0}^{r-1} M_{m_{\tau(j)}}
= \bigoplus_{k=1}^p M_{m_k}\otimes\mathbb C^{r_k}.
\]
Similarly,
\[
\bigoplus_{i=0}^{s-1} M_{n_{\sigma(i)}}
= \bigoplus_{k=1}^p \bigoplus_{i=1}^{s_k} M_{n_{k_i}}.
\]
Thus the inclusion decomposes as a direct sum of inclusions
\[
M_{m_k}\otimes\mathbb C^{r_k}
\subset
\bigoplus_{i=1}^{s_k} M_{n_{k_i}},
\qquad 1\le k\le p,
\]
with inclusion matrix $A_{kk}$.

By dimension counting,
$n_{k_i}=m_k r_k a_{k_i k^1}$, and hence
\[
\bigoplus_{i=1}^{s_k} M_{n_{k_i}}
=
\bigoplus_{i=1}^{s_k}
M_{a_{k_i k^1}}\otimes M_{m_k}\otimes M_{r_k}.
\]
Each of the inclusions
\[
M_{m_k}\subset M_{m_k}, \qquad
\mathbb C \subset \bigoplus_{i=1}^{s_k} M_{a_{k_i k^1}}, \qquad
\mathbb C^{r_k}\subset M_{r_k}
\]
is regular, with inclusion matrices 
\begin{equation*}
    \begin{bmatrix}
    1
\end{bmatrix},
\begin{bmatrix}
    a_{k_1k^1}&a_{k_2k^1}&\cdots&a_{k_{s_k}k^1}
\end{bmatrix}^t \text{and} \begin{bmatrix}
    1&1&\cdots&1
\end{bmatrix}
\end{equation*} respectively and therefore their tensor product
\[
M_{m_k}\otimes\mathbb C^{r_k}
\subset
\bigoplus_{i=1}^{s_k} M_{n_{k_i}}
\]
is regular whose inclusion matrix $A_{kk}$ is tensor product of above three inclusion matrices.
Since finite direct sums of regular inclusions are regular, the inclusion
$\mathcal B\subset\mathcal A$ is regular.

This completes the proof.
\end{proof}

\color{black}
\begin{rmrk}
    Assume $\mathcal{B}$ is commutative. Then the regularity of the inclusion $\mathcal B \subset \mathcal A$ is equivalent to the  inclusion matrix being a normalizer matrix. Indeed, in this case each matrix summand of $\mathcal{B}$ has  one-dimensional, that is, $m_j=1$ for all $j\in \{0,1,\cdots,r-1\}$. 
    Moreover, the converse of Proposition \ref{ppsn: regular homogeneous} also holds. More precisely, if $\mathcal{B}$ is homogeneous and all entries of the inclusion matrix are equal to $t$, then $\mathcal{B}\subset M_{n}(\mathbb C)$ is isomorphic to tensor product of regular inclusions 
    \[
    M_m(\mathbb{C})\subset M_m(\mathbb{C}), \quad \mathbb C ^r \subset M_r(\mathbb{C}) \quad \text{and}\quad \mathbb C \subset M_{t}(\mathbb{C}).
    \]
    
    where $n=mrt$ and $\mathcal{B}=\bigoplus_{j=0}^{r-1}M_m(\mathbb{C})$.
    \end{rmrk}

\section{Unitary Orthonormal Bases and Depth Two Structure}

 In this section, we give two consequences of the theory of regularity established in the previous section. One addresses the existence of a unitary orthonormal basis contained in the normalizer, while the other concerns the depth of an inclusion.

\subsection{Unitary Orthonormal basis} 

We now turn to the problem of the existence of unitary orthonormal bases for finite-dimensional inclusions, with particular emphasis on bases contained in the normalizer. Building on the combinatorial description of regularity developed in the previous section, we show that the existence of such bases is completely determined by the inclusion matrix, associated dimension data of simple summands and an additional spectral condition \cite{BB}. While the construction in \cite{BB} produces unitary orthonormal bases for several inclusions, these bases need not, in general, be contained in the normalizer. In contrast, our characterization reduces the problem to explicit combinatorial and spectral criteria and yields concrete construcions of unitary orthonormal bases in the normalizer. As an application, we obtain an affirmative solution to a finite-dimensional analogue of the Bakshi-Gupta conjecture \cite{BG}.

\noindent We recall the following result from \cite{BB}
\begin{thm}[\cite{BB}, Theorem 3.3]
    Let $(\mathcal{B\subseteq A},E)$ be an inclusion of finite dimensional von Neumann algebras having $U$- property. Then $E$ is the unique conditional expectation preserving the tracial state $\varphi$ given by \begin{equation} \label{trace}
\varphi (\oplus X_i)= \frac{1}{\sum _{i=0}^{s-1}n_i^2} \sum
_{i=0}^{s-1}n_i~\mbox {trace}~ (X_i), ~~X_i \in M_{n_i}, 0\leq i\leq
(s-1).\end{equation} 
\end{thm}

\begin{thm}\label{regular-basis}
Let $(\mathcal{B} \subset \mathcal{A}, E)$ be an inclusion of finite-dimensional von Neumann algebras, where $E$ denotes the unique state-preserving conditional expectation.  
Then the following statements are equivalent:
\begin{itemize}
  \item[(i)]  There exists a unitary orthonormal basis of $\mathcal{A}$ over $\mathcal{B}$ contained in the normaliser
  $\mathcal{N}_{\mathcal{A}}(\mathcal{B})$.
  \item[(ii)] The inclusion $\mathcal{B} \subset \mathcal{A}$ is regular and satisfies the spectral condition.
\end{itemize}

\iffalse
Let ($\mathcal{B}\subset \mathcal{A},E$) be an inclusion of finite-dimensional von Neumann algebra where $E$ is the unique conditional expectation preserving the tracial state $\varphi$ defined above. Then the inclusion is regular and satisfies the spectral condition if and only if it admits a unitary orthonormal basis contained in the normalizer. \fi
\end{thm}

\begin{prf} 
Assume first that the inclusion $\mathcal{B}\subset \mathcal{A}$ is regular.
By Theorem \ref{mainthm}, the inclusion matrix is an normalizer matrix and by Proposition \ref{ppsn2}, we have an isomorphism  
\[
 (\mathcal{B}\subset \mathcal{A})\cong(\tilde{\mathcal{B}}\subset\tilde{\mathcal{A}})
\]  
with inclusion matrix $A'=\text{diag}(A_{11},A_{22},\cdots,A_{pp})$  and the dimension vector of $\tilde{\mathcal{B}}$ is $$\begin{bmatrix}
      m_1&\ldots&m_1&\ldots&m_p&\ldots&m_p
      \end{bmatrix}^t$$ where each $m_k$ appears consecutively for $r_k$ times.\\
    The subalgerba system $(\tilde{\mathcal{B}}\subset \tilde{\mathcal{A}},\tilde{E})$ where $\tilde{E}$ is the unique conditional expectation preserving the state \begin{equation*}
\varphi\circ\theta^{-1}=\tilde\varphi (\oplus X_{\sigma(i)})= \frac{1}{\sum _{i=0}^{s-1}n_i^2} \sum
_{i=0}^{s-1}n_i~\mbox {trace}~ (X_i), ~~X_i \in M_{n_i}, 1\leq i\leq
(s-1).\end{equation*} on $\tilde{\mathcal{A}}$, where $\theta$ is the isomorphism defined in Proposition \ref{ppsn2}. \\
The inclusion matrix $A'$ satisfies the spectral condition as 
\begin{equation}\label{spcd}
    \sum_{i=1}^{s_k}a_{k_ik^1}n_{k_i}=dm_k \;\;\;   \text{for} \;\; 1\leq k \leq p.
\end{equation} 
 By dimension counting, we have $
 n_{k_i}=m_kr_ka_{k_ik^1}   \qquad   1\leq i \leq s_k$, so the above spectral condition becomes 
 \begin{equation*}
     d=r_k\sum_{i=1}^{s_k}a_{k_ik^1}^2\;\;\ \text{for}\;1\leq k\leq p.
 \end{equation*}   
The inclusion decomposes as a direct sum of systems $(M_{m_k}\otimes \mathbb{C}^{r_k} \subset \oplus_{i=1}^{s_k} M_{n_{k_i}},\tilde{E_k})$ with inclusion matrix $A_{kk}$ where $\tilde{E}_k$ is the unique conditional expectation satisfing the state \begin{equation*}\label{the state}
\tilde\varphi_k (\oplus X_i)= \frac{1}{\sum _{i=1}^{s_k}n_{k_i}^2} \sum
_{i=1}^{s_k}n_{k_i}~\mbox {trace}~ (X_i), ~~X_i \in M_{n_{k_i}}, 1\leq i\leq
s_k.\end{equation*}  for $1\leq k\leq p$ and satisfies the spectral condition (\cref{spcd}).(Note that $\tilde\varphi=\sum_{k=1}^p\tilde\varphi_i$)

 Each such inclusion further decomposes as a tensor product of
 \[
 (M_{m_k}\subset M_{m_k},Id),\qquad (\mathbb{C}\subset \oplus_{i=1}^{s_k}M_{a_{k_ik^1}},E_1) \qquad (\mathbb{C}^{r_k}\subset M_{r_k},E_2)
 \]  
with inclusion matrices 
\begin{equation*}
    \begin{bmatrix}
    1
\end{bmatrix},
\begin{bmatrix}
    a_{k_1k^1}&a_{k_2k^1}&\cdots&a_{k_{s_k}k^1}
\end{bmatrix}^t \text{and} \begin{bmatrix}
    1&1&\cdots&1
\end{bmatrix}
\end{equation*}

respecitively. The corresponding states are 
\[
\tilde\varphi_k^0(X)=\frac{\text{trace}~(Y)}{m_k},~~Y\in M_{m_k}~~ ; \qquad \tilde\varphi_k^2=\frac{\text{trace}~(Z)}{r_k},~~Z\in M_{r_k}~~;
\]
\[
\tilde\varphi_k^1 (\oplus X_i)= \frac{1}{\sum _{i=1}^{s_k}a_{k_ik^1}^2} \sum
_{i=1}^{s_k}a_{k_ik^1}~\mbox {trace}~ (X_i), ~~X_i \in M_{a_{k_ik^1}},
\] so that $\tilde\varphi_k=\tilde\varphi_k^0\otimes\tilde\varphi_k^1\otimes\tilde\varphi_k^2$.

 Each of these three inclusions admits unitary orthonormal basis in the normalizer (as mentioned in the preliminary section). Hence, their tensor product $$(M_{m_k}\otimes \mathbb{C}^{r_k} \subset \oplus_{i=1}^{s_k} M_{n_{k_i}},\tilde{E_k})$$
 admits unitary o.n.b. in the normalizer. 
Finally, as a direct sum of these inclusions $\tilde{\mathcal{B}}\subset \tilde{\mathcal{A}}$ admits a unitary orthonormal basis in the normalizer.\\
Converse of the above statement follows from the definition of regularity and Theorem \cref{BB1}.\qed \end{prf}

The following corollary is an analogue of Bakshi-Gupta conjecture \cite{BG} for finite dimensional von Neumann algebras.

\begin{crlre}\label{crlreglasgow}
    For the inclusion ($\mathcal{B}\subset \mathcal{A},E$) if either $\mathcal{B}$ or $\mathcal{A}$ is a matrix algebra, then the following are equivalent,
    \begin{enumerate}\itemsep0em \itemsep0em
        \item The inclusion is regular,
        \item It admits a unitary orthonormal basis contained in the normalizer.
    \end{enumerate}
\end{crlre}

\begin{prf}
If the inclusion admits a unitary orthonormal basis in the normalizer, then by the definition of basis, the inclusion is regular.\\
 Conversely, assume that the inclusion is regular. We consider two cases.\\
\textbf{Case 1:} $\mathcal{B}$ is a matrix algebra.\\
Then the inclusion is $M_m\subset \oplus_{i=0}^{s-1}M_{n_i}$ whose inclusion matrix is $ A=\begin{bmatrix}
        a_0&a_1& \ldots& a_{s-1}
    \end{bmatrix}^t$
Clearly, the above inclusion matrix satisfies the spectral condition and it is a \emph{normalizer matrix}. So, by the above theorem, it admits a unitary orthonormal basis in the normalizer.\\
\textbf{Case 2:} $\mathcal{A}$ is a matrix algebra.\\
Then the inclusion is $\oplus_{j=0}^{s-1}M_{m_j}\subset M_n$. Regularity implies the dimension vector of $\mathcal{B}$ as $m'=\begin{bmatrix}
    m&m& \ldots &m 
\end{bmatrix}^t$ and the inclusion matrix is $A=\begin{bmatrix}
    c&c&\cdots&c
\end{bmatrix}$ for some $c\in \mathbb{N}$. Clearly, it satisfies the spectral condition and so by the previous theorem it admits an unitary orthonormal basis in the normalizer. \qed
\end{prf}

\begin{crlre}\label{crlreglasgow1}
     For the inclusion $\mathbb C^r \subset M_n(\mathbb C)$ , the inclusion is regular if and only if the corresponding inclusion matrix satisfies spectral condition.
\end{crlre}

\begin{prf}
    Let $A=\begin{bmatrix}
        a_0&a_1&\cdots&a_{r-1}
    \end{bmatrix}$ denote the inclusion matrix. Suppose first that $A$ satisfies the spectral condition, then $A^tn'=dm'$, where $n'=\begin{bmatrix}
        n
    \end{bmatrix}$ and $m'=\begin{bmatrix}
        1&1&\cdots&1
    \end{bmatrix}^t$. This implies $a_jn=d \quad \forall j$. Consequently, all entries of $A$ are equal. Hence $A$ is a normalizer matrix and by Theorem \ref{mainthm}, the inclusion is regular. Conversely, if the inclusion is regular, then  $a_j=t  \quad \forall j$, for some $t\in \mathbb{N}$. Therefore, $A$ satisfies the spectral condition. This completes the proof. \qed
\end{prf}

  \begin{rmrk}
Combining the above two corollaries, suppose that the smaller algebra
$\mathcal B$ is commutative and that either $\mathcal A$ or $\mathcal B$ is a
matrix algebra. Then the following conditions are equivalent:
\begin{itemize}
  \item [(a)] the spectral condition on the inclusion matrix;
  \item [(b)] regularity of the inclusion $\mathcal B \subset \mathcal A$;
  \item [(c)] the existence of a unitary orthonormal basis contained in the
  normalizer of $\mathcal B$ in $\mathcal A$.
\end{itemize}
\end{rmrk}

\subsection{Depth of an inclusion}
For a finite group $G$ and a normal subgroup $H$, it is well known that the inclusion of the complex group algebra $\mathbb C (H)\subset \mathbb C (G)$ is regular (See example 1.2.3 \cite{W}). In this classical setting, the depth of the inclusion was computed in \cite{BKK} (Theorem 6.9) and shown to be equal to $2$. This example motivates a broader investigation of the depth of regular inclusion in finite-dimensional von Neumann algebras. In this section, we use the {normalizer matrix } associated with a regular inclusion to extend this result to general finite dimensional case.

 We now recall the notion of \emph{depth} as in \cite{BKK}

\begin{dfn}
    An $s\times r$ matrix $A$ is of depth $n\geq 2$ if $n$ is the least integer for which the following inequality (called a depth $n$ matrix inequality) holds for some $q\in \mathbb{Z}_+$,\begin{equation}
        A^{n+1}\leq qA^{n-1}.
    \end{equation} where the power of $s\times r$ matrix $A$ is understood as $A^2=AA^t, A^3=AA^tA$, and so forth.
    The definition depends only on the equivalence class of $A$ up to permutation.
\end{dfn}

In parts of the literature, the {\em inclusion matrix } is defined as the transpose of the inclusion matrix used in this article. In general, the depths associated with a matrix $A$ and its transpose $A^{t}$ need not coincide. However, when the depth is even, $A$ and $A^t$ have the same depth, see \cite{BKK} Theorem 3.16.

\begin{thm}
    If the inclusion $\mathcal{B}\subset \mathcal{A}$ is regular, then the corresponding inclusion matrix $A$ is of depth 2 and the constant $q$ appearing in the definition is bounded below by square of norm of the inclusion matrix.
\end{thm}

\begin{prf}
    By Theorem \ref{mainthm},  $A$ is pseudo-equivalent to \begin{equation*}
         A':= \text{diag}(A_{11},A_{22},\cdots,A_{pp})
    \end{equation*}
    So, $A'^3=\text{diag}(A_{11}^3,A_{22}^3,\cdots,A_{pp}^3)$
\end{prf}, where $A_{kk}^3=A_{kk}A_{kk}^tA_{kk}=(r_k\sum_{i=1}^{s_k} a_{k_ik^1}^2)A_{kk}$.\\
Choosing \begin{equation*}
    q=\text{max}\{r_k\sum_{i=1}^{s_k} a_{k_ik^1}^2:\;1\leq k\leq p\}=\lVert A \rVert^2,
\end{equation*} we have $A^3\leq qA$.  \qed

\iffalse
\begin{rmrk}
    Additionally, if the inclusion is connected as in \cite{GHJ} then $A^3=[\mathcal{A}:\mathcal{B}]A$.
\end{rmrk}
\fi
On more details on depth of an inclusion, we refer reader to \cite{BKK} and the bibliography therein.
%\begin{crlre}
%    For any $n\geq 2$ the standard inclusion $\mathbb{C}S_n\subset\mathbb{C} S_{n-1}$ is not regular.
%\end{crlre}

%    \subsection{Quantum chromatic number}
%    Let $\mathcal{B}\subseteq \mathcal{A}$ be a strongly Markov inclusion of finite-dimensional von Neumann algebras which admits an orthonormal Pimsner-Popa basis $\{u_i\}$ for $\mathcal{A}$ over $\mathcal{B}$ in $\mathcal{N_A(B)}$. Then \begin{equation}
 %   \mathcal{X}_{loc}(\mathcal{A},\mathcal{B}',B(L^2(\mathcal{M},\tau)))=\mathcal{X}_{q}(\mathcal{A},\mathcal{B}',B(L^2(\mathcal{M},\tau)))=\mathcal{X}_{qc}(\mathcal{A},\mathcal{B}',B(L^2(\mathcal{M},\tau)))=[\mathcal{A}:\mathcal{B}]= \lvert\lvert A\rvert\rvert^2.
%    \end{equation} 

\subsection*{Acknowledgement} 
The authors are grateful to Prof. B. V. Rajarama Bhat for helpful discussion. The first author acknowledges the support of the INSPIRE Faculty Grant DST/INSPIRE/04/2019/002754
and ANRF/ECRG/2024/002328/PMS.

\noindent {\em Department of Mathematics and Statistics},\\
{\em Indian Institute of Technology Kanpur},\\
{\em Uttar Pradesh $208016$, India}
\medskip
  
\noindent {Keshab Chandra Bakshi:} keshab@iitk.ac.in, bakshi209@gmail.com\\
{Silambarasan C:} silamc23@iitk.ac.in

\end{document}